\documentclass{amsart}

\usepackage{latexsym}
\usepackage{amsthm,amsfonts,amssymb,amscd,mathrsfs}
\usepackage{xspace}
\input{xypic}
\xyoption{all}

\newcommand{\numberseries}{\mdseries}	

\newlength{\thmtopspace}		
\newlength{\thmbotspace}		
\newlength{\thmheadspace}		
\newlength{\thmindent}			

\newtheoremstyle{bfupright head,slanted body}
		{\thmtopspace}{\thmbotspace}
		{\slshape}{\thmindent}{\bfseries}{.}{\thmheadspace}
		{{\numberseries \thmnumber{(#2) }}\thmnote{#3}}

\newtheoremstyle{bfupright head,upright body}
		{\thmtopspace}{\thmbotspace}
		{\upshape}{\thmindent}{\bfseries}{.}{\thmheadspace}
		{{\numberseries \thmnumber{(#2) }}\thmnote{#3}}

\newtheoremstyle{bfit head,upright body}
		{\thmtopspace}{\thmbotspace}
		{\upshape}{\thmindent}{\upshape}{.}{\thmheadspace}
		{{\numberseries\thmnumber{(#2) }}
		{\bfseries\itshape\thmnote{\negthickspace#3}}}

\newtheoremstyle{it head,upright body}
		{\thmtopspace}{\thmbotspace}
		{\upshape}{\thmindent}{\upshape}{.}{\thmheadspace}
		{{\numberseries\thmnumber{(#2) }}
		{\itshape\thmnote{\negthickspace#3}}}

\newtheoremstyle{fixed bf head,slanted body}
		{\thmtopspace}{\thmbotspace}{\slshape}
		{\thmindent}{\bfseries}{.}{\thmheadspace}
		{{\numberseries \thmnumber{(#2) }}\thmname{#1}\thmnote{ (#3)}}

\newtheoremstyle{fixed bf head,upright body}
		{\thmtopspace}{\thmbotspace}{\upshape}
		{\thmindent}{\bfseries}{.}{\thmheadspace}
		{{\numberseries \thmnumber{(#2) }}\thmname{#1}\thmnote{ (#3)}}

\newtheoremstyle{indented paragraph}
		{\thmtopspace}{\thmbotspace}
		{\upshape}{\thmindent}{\upshape}{}{0pt}
		{\thmnote{#3 }}

\theoremstyle{bfupright head,slanted body}
\newtheorem{res}{}[section]		\newtheorem*{res*}{}

\theoremstyle{bfit head,upright body}
			\newtheorem*{com*}{}

\theoremstyle{bfupright head,upright body}
\newtheorem{bfhpg}[res]{}		\newtheorem*{bfhpg*}{}

\theoremstyle{it head,upright body}
		\newtheorem*{ithpg*}{}

\theoremstyle{fixed bf head,slanted body}
\newtheorem{thm}[res]{Theorem}		\newtheorem*{thm*}{Theorem}
\newtheorem{prp}[res]{Proposition}	\newtheorem*{prp*}{Proposition}
\newtheorem{cor}[res]{Corollary}	\newtheorem*{cor*}{Corollary}
\newtheorem{lem}[res]{Lemma}		\newtheorem*{lem*}{Lemma}

\theoremstyle{fixed bf head,upright body}
\newtheorem{dfn}[res]{Definition}	\newtheorem*{dfn*}{Definition}
\newtheorem{obs}[res]{Observation}	\newtheorem*{obs*}{Observation}
\newtheorem{rmk}[res]{Remark}		\newtheorem*{rmk*}{Remark}
\newtheorem{exa}[res]{Example}		\newtheorem*{exa*}{Example}
		\newtheorem*{exe*}{Exercise}
		\newtheorem{stp*}{Setup}
	\newtheorem*{dfns*}{Definitions}
	\newtheorem*{obss*}{Observations}
		\newtheorem*{rmks*}{Remarks}
	\newtheorem*{exas*}{Examples}

\theoremstyle{indented paragraph}

\newlength{\thmlistleft}	
\newlength{\thmlistright}	
\newlength{\thmlistpartopsep}	
\newlength{\thmlisttopsep}	
\newlength{\thmlistparsep}	
\newlength{\thmlistitemsep}	

\newcounter{eqc} 
\newenvironment{eqc}{\begin{list}{\upshape (\textit{\roman{eqc}})}%
		    {\usecounter{eqc}%
		        \setlength{\leftmargin}{\thmlistleft}%
			\setlength{\labelwidth}{\thmlistleft}%
			\setlength{\rightmargin}{\thmlistright}%
			\setlength{\partopsep}{\thmlistpartopsep}%
			\setlength{\topsep}{\thmlisttopsep}%
			\setlength{\parsep}{\thmlistparsep}%
			\setlength{\itemsep}{\thmlistitemsep}}}%
		    {\end{list}}%

\newcounter{prt}
\newenvironment{prt}{\begin{list}{\upshape (\alph{prt})}%
		    {\usecounter{prt}%
		        \setlength{\leftmargin}{\thmlistleft}%
			\setlength{\labelwidth}{\thmlistleft}%
			\setlength{\rightmargin}{\thmlistright}%
			\setlength{\partopsep}{\thmlistpartopsep}%
			\setlength{\topsep}{\thmlisttopsep}%
			\setlength{\parsep}{\thmlistparsep}%
			\setlength{\itemsep}{\thmlistitemsep}}}%
		    {\end{list}}%

\newcounter{rqm}
\newenvironment{rqm}{\begin{list}{\upshape (\arabic{rqm})}%
		    {\usecounter{rqm}%
		        \setlength{\leftmargin}{\thmlistleft}%
			\setlength{\labelwidth}{\thmlistleft}%
			\setlength{\rightmargin}{\thmlistright}%
			\setlength{\partopsep}{\thmlistpartopsep}%
			\setlength{\topsep}{\thmlisttopsep}%
			\setlength{\parsep}{\thmlistparsep}%
			\setlength{\itemsep}{\thmlistitemsep}}}%
		    {\end{list}}%

			{\end{list}}%

\newenvironment{itemlist}{\nopagebreak \begin{list}{$\bullet$}%
		       {\setlength{\leftmargin}{\thmlistleft}%
			\setlength{\labelwidth}{\thmlistleft}%
			\setlength{\rightmargin}{\thmlistright}%
			\setlength{\partopsep}{\thmlistpartopsep}%
			\setlength{\topsep}{\thmlisttopsep}%
			\setlength{\parsep}{\thmlistparsep}%
			\setlength{\itemsep}{\thmlistitemsep}}}%
			{\end{list}}%

			{\end{list}}%
\newlength{\myindent}
{\setlength{\myindent}{\parindent}\begin{list}{}%
			{\setlength{\leftmargin}{#1}\setlength{\rightmargin}{#1}%
			\setlength{\partopsep}{0pt}%
			\setlength{\topsep}{\thmtopspace}%
			\setlength{\parsep}{0pt}%
			\setlength{\itemsep}{0pt}}
			\item[]}
			{\end{list}}%

\newenvironment{proof*}{\begin{proof}}{\renewcommand{\qed}{} \end{proof}}

\newlength{\seqsplit}
\title[Modules with cosupport and injective functors]{Modules with
  cosupport and injective functors}

\author{Henrik Holm} 

\address{\begin{flushleft} Department of Natural Sciences,
  Faculty of Life Sciences, University of Copenhagen, Thorvaldsensvej
  40, 6th floor, DK-1871 Frederiksberg C, Denmark \end{flushleft}}

\email{hholm@life.ku.dk} 

\urladdr{http://www.dina.life.ku.dk/$\sim$hholm}

 

\newcommand{\Ker}{\operatorname{Ker}}
\newcommand{\Hom}{\operatorname{Hom}}
\newcommand{\Ext}{\operatorname{Ext}}
\newcommand{\Tor}{\operatorname{Tor}}

\newcommand{\Ab}{\mathsf{Ab}}
\newcommand{\proj}{\mathsf{proj}}
\newcommand{\Flat}{\mathsf{Flat}}
\newcommand{\Inj}{\mathsf{Inj}}
\newcommand{\PureInj}{\mathsf{PureInj}}

\newcommand{\add}{\mathsf{add}}
\newcommand{\Add}{\mathsf{Add}}
\newcommand{\Prod}{\mathsf{Prod}}
\renewcommand{\mod}{\mathsf{mod}}
\newcommand{\RMod}[1][R]{#1\textnormal{-}\mathsf{Mod}}
\newcommand{\ModR}[1][R]{\mathsf{Mod}\textnormal{-}#1}

\newcommand{\ZZ}{\mathbb{Z}}
\newcommand{\QZ}{\mathbb{Q}/\mathbb{Z}}

\newcommand{\B}{\mathcal{B}}
\newcommand{\C}{\mathcal{C}}
\newcommand{\E}{\mathcal{E}}
\newcommand{\M}{\mathcal{M}}

\newcommand{\Rop}{R^\mathrm{op}}
\newcommand{\Bop}{\B^\mathrm{op}}

\newcommand{\lora}{\longrightarrow}

\begin{document}

\begin{abstract} 
  Several authors have studied the filtered colimit closure
  $\varinjlim\B$ of a class $\B$ of finitely presented modules.
  Lenzing called $\varinjlim\B$ the category of modules with support
  in $\B$, and proved that it is equivalent to the category of flat
  objects in the functor category $(\Bop,\Ab)$.  In this paper, we
  study the category $(\ModR)^\B$ of modules with cosupport in $\B$.
  We show that $(\ModR)^\B$ is equivalent to the category of injective
  objects in $(\B,\Ab)$, and thus recover a classical result by
  Jensen-Lenzing on pure injective modules.  Works of
  Angeleri-H{\"u}gel, Enochs, Krause, Rada, and Saor{\'{\i}}n make it
  easy to discuss covering and enveloping properties of $(\ModR)^\B$,
  and furthermore we compare the naturally associated notions of
  $\B$-coherence and $\B$-noetherianness. Finally, we prove a number
  of stability results for $\varinjlim\B$ and $(\ModR)^\B$. Our
  applications include a generalization of a result by Gruson-Jensen
  and Enochs on pure injective envelopes of flat modules.
\end{abstract}

\maketitle


\section*{Introduction}

Let $B$ be a finitely presented left module over a ring $R$, and let
$\Lambda$ be its endo\-morphism ring. Since $B$ is a
left-$\Lambda$-left-$R$-bimodule, one can consider the functors
\begin{displaymath}
  \xymatrix@C=20ex{
    \RMod \, \ar@<0.7ex>[r]^-{\Hom_R^{}(B,-)} &
    \ \ModR[\Lambda]. 
    \ar@<0.7ex>[l]^-{-\otimes^{}_{\!\Lambda} B} 
  }
\end{displaymath}

An important observation in Auslander's work on representation theory
for Artin algebras is that these functors give an equivalence between
$\add\,B$ and $\proj\text{-}\Lambda$; see notation in
\eqref{bfhpg:notation}.  Actually, it follows by Lazard \cite{DL64}
that the functors above also induce an equivalence between
$\varinjlim(\add\,B)$ and $\Flat\text{-}\Lambda$. In \cite{HL83}
Lenzing generalizes this result even further by proving that for any
additive category $\B$ of finitely presented left $R$-modules, the
Yoneda functor,
\begin{displaymath}
  \RMod \lora (\Bop,\Ab) \quad ,\quad M \longmapsto
  \Hom_R(-,M)|_\B
\end{displaymath}
restricts to an equivalence between $\varinjlim\B$ and the category
$\Flat(\Bop,\Ab)$ of flat functors in the sense of Oberst-R{\"o}hrl
\cite{OR70} and Stenstr{\"o}m \cite{BS68}.  The category
$\varinjlim\,\B$ has several nice properties, and it has been studied
in great detail by e.g.~the authors of \cite{LAH03},
\cite{LAH-JS-JT-08}, \cite{LAH-JT-04}, \cite{WCB94}, \cite{REB06},
\cite{HK-OS-03}, and \cite{HL83}.

In this paper, we study the category of modules with
\textsl{cosupport} in $\B$,
\begin{displaymath}
  (\ModR)^\B = \Prod\{
  \Hom_\ZZ(B,\QZ) \,|\, B \in \B \}.
\end{displaymath}
The main theorem of Section \ref{sec:InjFunc} is a result dual to that
of Lenzing \cite[prop.~2.4]{HL83}.
\begin{res*}[Theorem A]
  The tensor embedding (which is not necessarily an embedding),
  \begin{displaymath}
    \ModR \lora (\B,\Ab) \quad ,\quad N \longmapsto
    (N\otimes_R-)|_\B
  \end{displaymath}
  restricts to an equivalence between $(\ModR)^\B$ and $\Inj(\B,\Ab)$.
\end{res*}

Two special cases of Theorem A are worth mentioning: If \mbox{$\B =
  \add\,B$} for some finitely presented module $B$ with endomorphism
ring $\Lambda$, it follows that the functors
\begin{displaymath}
  \xymatrix@C=20ex{
    \ModR \, \ar@<0.7ex>[r]^-{-\otimes^{}_{\!R} B} & \ \RMod[\Lambda]
    \ar@<0.7ex>[l]^-{\Hom^{}_{\Lambda}(B,-)} 
  }
\end{displaymath}
induce an equivalence between $\Prod\{ \Hom_\ZZ(B,\QZ) \}$ and
$\Lambda\text{-}\Inj$. For \mbox{$\B=R\text{-}\mod$} we get an
equivalence between the category of pure injective right $R$-modules
and $\Inj(R\text{-}\mod,\Ab)$. We refer to 
Jensen-Lenzing \cite[thm.~B.16]{CUJ-HL-book}\footnote{\ Unfortunately,
  the proof of Jensen-Lenzing \cite[thm.~B.16]{CUJ-HL-book} does not
  apply to give a proof of Theorem A, as one key ingredient in their
  argument is the fact that the tensor embedding
  \begin{displaymath}
    \ModR \lora (R\text{-}\mod,\Ab) \quad ,\quad N \longmapsto
    (N\otimes_R-)|_{R\text{-}\mod}
  \end{displaymath}
  is fully faithful. If \mbox{$R \notin \B$} the tensor ``embedding''
  in Theorem A is, in general, neither full nor faithful as Example
  \eqref{exa:not_full_not_faithful} shows. Our proof of Theorem A uses
  techniques, such as Eilenberg's swindle and tensor products of
  functors, different from those found in the proof of
  \cite[thm.~B.16]{CUJ-HL-book}.} for this classical result.

In Section \ref{sec:EnvCov} we investigate enveloping and covering
properties of $(\ModR)^\B$.  One easy consequence of
Theorem A is the following:
\begin{res*}[Theorem B]
  The class $(\ModR)^\B$ is enveloping in $\ModR$. In addition, for a
  ho\-mo\-mor\-phism \mbox{$h \colon N \lora I$} with $I$ in
  $(\ModR)^\B$, the following conditions are equivalent:
  \begin{eqc}
  \item $h$ is a $(\ModR)^\B$-envelope;
  \item $h$ is $\B$-essential $\B$-monomorphism, cf.~Definition
    \eqref{dfn:essential}.
  \end{eqc}
\end{res*}

Theorem B is not new, but it does cover several references in the
literature: That $(\ModR)^\B$ is enveloping also follows from
Enochs-Jenda-Xu \cite[thm.~2.1]{EEE-OMGJ-JX-93} and Krause
\cite[cor.~3.15]{HK01}. In the case where $R$ is in $\varinjlim\B$,
the class of short \mbox{$(-\otimes_R\B)$}-exact sequences constitutes
a \textsl{proper class} in the sense of Stenstr{\"o}m
\cite[\S2]{BTS67}, and hence Theorem B also contains
\cite[prop.~4.5]{BTS67}.

We stress that the hard parts of the proof of Theorem~C below follow
from refe\-rences to works of Angeleri-H{\"u}gel, Krause, Rada
and Saor{\'{\i}}n, \cite{LAH00}, \cite{HK01}, \cite{HK-MS-98},
\cite{JR-MS-98}.

\enlargethispage{3ex}

\begin{res*}[Theorem C]
  For the full subcategory $(\ModR)^\B$ of $\ModR$, the following
  conditions are equivalent:
  \begin{eqc}
  \item It is closed under coproducts;
  \item It is closed under direct limits;
  \item It is precovering;
  \item It is covering;
  \item It is closed under pure submodules;
  \item It is closed under pure submodules, pure quotients, and pure
    extensions;
  \item It equals $\Add\,E$ for some right $R$-module $E$.
  \end{eqc}
\end{res*}

If the conditions in Theorem C are satisfied, $R$ is called
\textsl{$\B$-noetherian}. In \eqref{dfn:Bcoherent} we define what it
means for $R$ to be \textsl{$\B$-coherent}. Using this terminology, we
give in \eqref{prp:cotorsion} a criterion for the existence of a
cotorsion pair $(\M,(\ModR)^\B)$ of finite type.

In Section \ref{sec:stability} we prove stability results for modules
with (co)support in $\B$, e.g.

\begin{res*}[Theorem D]
  A module $F$ is in $\varinjlim\B$ if and only if $\,\Hom_{\ZZ}(F,\QZ)$
  is in $(\ModR)^\B$.
\end{res*}

\begin{res*}[Theorem E]
  Assume that $R$ is in $\varinjlim\B$. Then $R$ is $\B$-noetherian if
  and only if
  \begin{rqm}
  \item $R$ is $\B$-coherent, and
  \item Any right $R$-module $E$ is in $(\ModR)^\B$ if only if
    $\,\Hom_\ZZ(E,\QZ)$ is in $\varinjlim\B$.
  \end{rqm}
\end{res*}

\noindent
We point out a couple of applications of the stability theorems above:

Corollary \eqref{cor:cap} gives conditions on a class $\E$ which
ensure that \mbox{$\E \cap \PureInj\text{-}R$} has the form
$(\ModR)^\B$. In Example \eqref{exa:Iwanaga} we apply \eqref{cor:cap}
to describe the modules with cosupport in the category of
$G$-dimension zero modules over a Gorenstein ring.

Corollary \eqref{cor:Lenzing-like} describes some new properties for
the class $\varinjlim\B$ of modules with support in $\B$. These
properties are akin to those found in Lenzing \cite[\S2]{HL83}.

Corollary \eqref{cor:PE} generalizes a result by Gruson-Jensen
\cite{LG-CUJ-80} and Enochs \cite{EEE87} which asserts that over a
coherent ring, the pure injective envelope of a flat module is flat.

The paper ends with Appendix \ref{sec:app} where we show two results
on injective and flat functors. These results are needed to prove the
stability theorems in Section \ref{sec:stability}.

\enlargethispage{7ex}

\section{Preliminaries}

In this preliminary section, we introduce our notation, define modules
with cosuppport in $\B$, and briefly present some relevant background
material.

\begin{bfhpg}[Setup]
  \label{bfhpg:setup}
  Throughout this paper, $R$ is any unital ring and $\B$ denotes any
  additive full subcategory of the category of finitely presented left
  $R$-modules.
\end{bfhpg}

\begin{bfhpg}[Notation]
  \label{bfhpg:notation}
  We write $\RMod$/$\ModR$ for the category of left/right $R$-modules,
  and $\Ab$ for the category of abelian groups. As in Krause-Solberg
  \cite{HK-OS-03}, we define for \mbox{$\C \subseteq \RMod$} four full
  subcategories of $\RMod$ by specifying their objects as below.
  \begin{itemlist}
  \item $\add\,\C$ -- direct summands of finite (co)products
    of modules from $\C$;
  \item $\Add\,\C$ -- direct summands of arbitrary coproducts
    of modules from $\C$;
  \item $\Prod\,\C$ -- direct summands of arbitrary products
    of modules from $\C$;
  \item $\varinjlim \C$ -- filtered colimits,
    cf.~\cite[IX.\S1]{SML-book}, of modules from $\C$.
  \end{itemlist}
  Some authors \cite{LAH-JS-JT-08}, \cite{LAH-JT-04}, \cite{HK-OS-03}
  use the notation $\varinjlim \C$---others \cite{LAH03},
  \cite[\S4]{WCB94} write $\vec{\C}$.  The following specific
  categories of modules play a central role in our examples.
   \begin{itemlist}
   \item $\mod$ -- finitely presented modules;
   \item $\proj$ -- finitely generated projective modules;
   \item $\Flat$ -- flat modules;
   \item $\Inj$ -- injective modules;
   \item $\PureInj$ -- pure injective modules.
   \end{itemlist}
\end{bfhpg}

\begin{dfn}
  \label{dfn:cosupport}
  Modules \textsl{with support in $\B$} was defined by Lenzing
  \cite{HL83},
  \begin{displaymath}
    (\RMod)_\B = \varinjlim\B.
  \end{displaymath}
  In this paper we study the category of right $R$-modules
  \textsl{with cosupport in $\B$},
  \begin{displaymath}
    (\ModR)^\B = \Prod\{ \Hom_\ZZ(B,\QZ) \,|\, B \in \B \}.
  \end{displaymath}
\end{dfn}

\begin{exa}
  \label{exa:ordinary_inj}
  The following is well-known.
  \begin{prt}
  \item If $\B=R\text{-}\proj$ then $(\ModR)^\B=\Inj\text{-}R$. 
  \item If $\B=R\text{-}\mod$ then $(\ModR)^\B=\mathsf{PureInj}\text{-}R$. 
  \end{prt}
\end{exa}

\begin{exa}
  \label{exa:C}
  Let $R$ be commutative and noetherian, let $C$ be a
  semiduali\-zing\footnote{%
    \ A finitely generated module is semidualizing if the homothety
    map $R \lora \Hom_R(C,C)$ is an isomorphism.  Semidualizing
    modules have been studied under different names by
    Foxby~\cite{foxby:gmarm} (PG-modules of rank one),
    Golod~\cite{golod:gdgpi} (suitable modules), and
    Vasconcelos~\cite{vasconcelos:dtmc} (spherical modules).}
  $R$-module, and let \mbox{$\B=\add\,C$}. Combining Example
  \eqref{exa:ordinary_inj}(a) with the isomorphism
  \begin{displaymath}
    \Hom_\ZZ(C,\QZ) \cong \Hom_R(C,\Hom_\ZZ(R,\QZ)),
  \end{displaymath}
  it is easily seen that $(\ModR)^\B$ consists exactly of modules of
  the form $\Hom_R(C,E)$, where $E$ is injective.  These modules play
  a central role in e.g.~\cite{EEE-SY-04}, \cite{EEE-OMGJ-96},
  \cite{HH-PJ-06}.
\end{exa}

\begin{exa}
  \label{exa:Iwanaga}
  Assume that $R$ is Iwanaga-Gorenstein, that is, $R$ is two-sided
  noetherian and has finite injective dimension from both sides.
  Consider:
  \begin{itemlist}
  \item[--] The class $\B$ of $G$-dimension zero\footnote{\ A
      f.g.~$R$-module $B$ is of $G$-dimension zero if $\Ext^{\geqslant
        1}(B,R)=0=\Ext^{\geqslant 1}(\Hom(B,R),R)$ and if the
      biduality homomorphism $B \lora \Hom(\Hom(B,R),R)$ is an
      isomorphism.} left $R$-modules, cf.~Auslander-Bridger
    \cite{MA-MB-69};
  \item[--] The class $\E$ of Gorenstein injective\footnote{\ $M$ is
      Gorenstein injective if there is an exact sequence
      $\boldsymbol{E} = \cdots \to E_1 \to E_0 \to E_{-1} \to \cdots$
      of injective modules such that $\Hom(I,\boldsymbol{E})$ is exact
      for all injective $I$ and $M \cong \Ker(E_0 \to E_{-1})$.} right
    $R$-modules, cf.~Enochs-Jenda \cite{EEE-OMGJ-95}.
  \end{itemlist}
  Then there is an equality, $(\ModR)^\B = \E \cap \PureInj\text{-}R$.
\end{exa}

\begin{proof}
  We apply Corollary \eqref{cor:cap}. By \cite[thm.~2.6]{HH04} the
  class $\E$ is closed under products and direct summands.  Condition
  \eqref{cor:cap}(1) holds by \cite[prop.~3.8]{LWC-AF-HH-06} and
  \cite[thm.~3.6]{HH04}; and condition \eqref{cor:cap}(2) holds by
  \cite[cor.~10.3.9]{EEE-OMGJ-book} and
  \cite[thm.~10.3.8]{EEE-OMGJ-book}.
\end{proof}

\begin{bfhpg}[Functor categories]
  \label{bfhpg:functor_categories}
  Let $\C$ be any additive and and skeletally small category, for
  example \mbox{$\C=\B$} from Setup \eqref{bfhpg:setup}. We adopt the
  notation of \cite{WCB94}, \cite{HK-OS-03} and write $(\C,\Ab)$ for
  the category of all additive covariant functors \mbox{$\C \lora
    \Ab$}.

  It is well-known, cf.~\cite[II.\S1]{PG62} that $(\C,\Ab)$ is an
  abelian category with small Hom-sets, and that $(\C,\Ab)$ admits the
  same categorical constructions (such as exact direct limits) as
  $\Ab$ does.  The representable functors $\C(C,-)$ are projective
  objects, and they constitute a generating set. Thus $(\C,\Ab)$ has
  injective hulls in the sense of \cite[II.\S5, \S6]{PG62}. We write
  $\Inj(\C,\Ab)$ for the category of injective objects in $(\C,\Ab)$.

  A functor $F$ is \textsl{finitely generated} if there is an exact
  sequence \mbox{$\C(C,-) \to F \to 0$} for some \mbox{$C \in
    \C$}. Similarly, $F$ is \textsl{finitely presented} if there
  exists an exact sequence \mbox{$\C(C_1,-) \to \C(C_0,-) \to F
    \to 0$} with $C_0,C_1 \in \C$.
\end{bfhpg}


\begin{bfhpg}[Flat functors]
  \label{bfhpg:flat_functors}
  Oberst-R{\"o}hrl \cite[\S1]{OR70} and Stenstr{\"o}m \cite[\S3]{BS68}
  construct over any pre\-additive and skeletally small category $\C$ a
  right exact tensor product,
  \begin{displaymath}
    (\C^\mathrm{op},\Ab) \times (\C,\Ab) \lora \Ab \quad , \quad (F,G)
    \longmapsto F \otimes_\C G
  \end{displaymath}
  which has the following properties for all $F$ and $G$ as above, and
  all $A \in \Ab$.
  \begin{prt}
  \item $\Hom_\ZZ(F \otimes_\C G,A) \cong (\C,\Ab)(G,\Hom_\ZZ(F,A))
      \cong (\C^\mathrm{op},\Ab)(F,\Hom_\ZZ(G,A))$.
  \item $F\otimes_\C\C(C,-) \cong FC$ and $\C(-,C)\otimes_\C G
    \cong GC$.
  \end{prt}
  A functor $F$ in \mbox{$(\C^\mathrm{op},\Ab)$} is \textsl{flat} if
  \mbox{$F\,\otimes_\C\,?$} is exact, however,
  \cite[thm.~(1.3)]{WCB94}, \cite[thm.~(3.2)]{OR70}, and
  \cite[thm.~3]{BS68} contain several equivalent characterizations of
  flatness.

  We write $\Flat(\C^\mathrm{op},\Ab)$ for the category of flat
  functors in $(\C^\mathrm{op},\Ab)$.
\end{bfhpg}


\section{An equivalence between two categories.} \label{sec:InjFunc}

In this section, we prove that the category $(\ModR)^\B$ of modules
with cosupport in $\B$ is equivalent to the category of injective
objects in the functor category $(\B,\Ab)$.

\begin{dfn}
  \label{dfn:TA}
  The \textsl{tensor embedding} with respect to $\B$ is the following
  functor,
  \begin{displaymath}
    \ModR \lora (\B,\Ab) \quad ,\quad N \longmapsto
    (N\otimes_R-)|_\B.
  \end{displaymath}
\end{dfn}

\begin{rmk}
  For \mbox{$\B=R\text{-}\mod$} the tensor embedding has been studied
  in e.g.~\cite{MA78}, \cite{GG04}, \cite{CUJ-HL-book}, \cite{HK01}.
  In this case, the tensor embedding is fully faithful as the inverse
  of
  \begin{displaymath}
    \Hom_{\Rop}(M,N) \stackrel{\cong}{\lora}
    (\B,\Ab)\big((M\otimes_R-)|_\B, (N\otimes_R-)|_\B\big)
  \end{displaymath}
  is given by evaluating a natural transformation on the ground ring
  $R$.
\end{rmk}

\begin{exa}
  \label{exa:not_full_not_faithful}
  For general $\B$, the tensor ``embedding'' is neither full nor
  faithful.  To see this, let $R=\ZZ$, let $p \neq q$ be prime
  numbers and set $\B=\add\,\ZZ/(p)$.
  \begin{prt}
  \item As \mbox{$\ZZ/(p)\otimes_\ZZ\ZZ/(p) \cong \ZZ/(p)$} the
    functors $(\ZZ/(p)\otimes_\ZZ-)|_\B$ and $(\ZZ\otimes_\ZZ-)|_\B$
    are equivalent, and since $\Hom_{\ZZ}(\ZZ/(p),\ZZ)\cong 0$, the tensor
    embedding is not full.
  \item As \mbox{$\ZZ/(q)\otimes_\ZZ\ZZ/(p) \cong 0$} we get
    \mbox{$(\ZZ/(q)\otimes_\ZZ-)|_\B=0$}, so from the isomorphism
    \mbox{$\Hom_{\ZZ}(\ZZ/(q),\ZZ/(q))\cong \ZZ/(q)$}, the tensor
    embedding cannot be faithful.
  \end{prt}
\end{exa}

Part (d) of the next result shows that the tensor embedding does
become fully faithful when appropriately restricted.

\begin{prp}
  \label{prp:formula}
  The following conclusions hold:
  \begin{prt}
  \item The tensor embedding, cf.~Definition \eqref{dfn:TA}, is
    additive and commutes with small filtered colimits and products.
  \item For $B \in \B$ there is a natural equivalence of functors $\B
    \lora \Ab$,
    \begin{displaymath}
      (\Hom_\ZZ(B,\QZ)\otimes_R-)|_\B \,\simeq\, \Hom_\ZZ(\B(-,B),\QZ).
    \end{displaymath}    
  \item For $F \in (\B,\Ab)$ and $B \in \B$ there is a natural
    isomorphism of abelian groups,
  \begin{displaymath}
    \Hom_\ZZ(FB,\QZ) \stackrel{\cong}{\lora}
    (\B,\Ab)\big(F,(\Hom_\ZZ(B,\QZ)\otimes_R-)|_\B\big).
  \end{displaymath}
\item Let $N$ and $I$ be right $R$-modules where $I \in (\ModR)^\B$.
  The homomorphism of abelian groups induced by the tensor
  embedding is then an isomorphism,
  \begin{displaymath}
    \Hom_{\Rop}(N,I) \stackrel{\cong}{\lora}
    (\B,\Ab)\big((N\otimes_R-)|_\B,(I\otimes_R-)|_\B\big).
  \end{displaymath}
  \end{prt}
\end{prp}

\begin{proof}
  ``(a)'': Clearly, the tensor embedding is additive. It commutes with
  filtered colimits by \cite[cor.~2.6.17]{CAW-book}, and with products
  by \cite[thm.~3.2.22]{EEE-OMGJ-book}.

  ``(b)'': As $\B$ consists of finitely presented modules,
  \cite[prop.~VI.5.3]{CE56} gives that
  \begin{align*}
     (\Hom_\ZZ(B,\QZ)\otimes_R-)|_\B &\simeq
     \Hom_\ZZ(\Hom_R(-,B),\QZ)|_\B \\
     &\simeq \Hom_\ZZ(\B(-,B),\QZ).
  \end{align*}

  ``(c)'': By part (b) we get the first isomorphism in:
  \begin{align*}
    (\B,\Ab)(F,(\Hom_\ZZ(B,\QZ)\otimes_R-)|_\B)
     &\cong
    (\B,\Ab)(F,\Hom_\ZZ(\B(-,B),\QZ)) \\
     &\cong
    \Hom_\ZZ(\B(-,B)\otimes_\B F,\QZ) \\
     &\cong
    \Hom_\ZZ(FB,\QZ).
  \end{align*}
  The second and third isomorphisms are by
  \eqref{bfhpg:flat_functors}(a) and (b), respectively.

  ``(d)'': By Definition \eqref{dfn:cosupport}, $I$ is a direct
  summand of a product of modules of the form $\Hom_\ZZ(B,\QZ)$ where
  $B \in \B$.  Thus, since the tensor embedding and the covariant
  Hom-functors $\Hom_{\Rop}(N,?)$ and $(\B,\Ab)((N\otimes_R-)|_\B,?)$
  are all additive and commutes with products, we may assume that
  $I=\Hom_\ZZ(B,\QZ)$ with \mbox{$B \in \B$}. We then apply part (c)
  with $F=(N\otimes_R-)|_\B$ to get the first isomorphism in:
  \begin{align*}
    (\B,\Ab)((N\otimes_R-)|_\B,(I\otimes_R-)|_\B) 
     &\cong
     \Hom_\ZZ(N\otimes_RB,\QZ) \\
      &\cong
    \Hom_{\Rop}(N,I).
  \end{align*}
  The second isomorphism is by adjunction \cite[prop.~II.5.2]{CE56}
  and by definition of $I$.
\end{proof}

\begin{dfn}
  \label{dfn:Bmono}
  A homomorphism of right $R$-modules \mbox{$h \colon M \lora N$} is
  called a \textsl{$\B$-mono\-mor\-phism} if $h \otimes_RB$ is a
  monomorphism for all $B$ in $\B$.
\end{dfn}

\begin{lem}
  \label{lem:Bsplit}
  If $I$ has cosupport in $\B$ and \mbox{$h \colon I \lora N$} is a
  $\B$-monomorphism, then $h$ is a split monomorphism.
\end{lem}

\begin{proof}
  By our assumptions and by the isomorphism,
  \begin{displaymath}
    \Hom_{\Rop}(h,\Hom_\ZZ(B,\QZ)) \cong
    \Hom_\ZZ(h \otimes_R B, \QZ),
  \end{displaymath}
  it follows that $\Hom_{\Rop}(h,\Hom_\ZZ(B,\QZ))$ is surjective for
  every $B$ in $\B$. Combining this with Definition
  \eqref{dfn:cosupport}, we see that $\Hom_{\Rop}(h,J)$ is surjective
  for all $J$ with cosupport in $\B$, that is, every homomorphism
  \mbox{$I \lora J$} factors through $h$. If $I$ has cosupport in
  $\B$, we apply this to \mbox{$\mathrm{id} \colon I \lora I$} to get
  the desired conclusion.
\end{proof}

Once we have proved Theorem A, the following Lemmas
\eqref{lem:embedding} and \eqref{lem:summand} will be superfluous.
These lemmas are the key ingredients in proving essential sujectivity
of the tensor embedding when viewed as a functor from $(\ModR)^\B$ to
$\Inj\,(\B,\Ab)$.


\begin{lem}
  \label{lem:embedding}
  Every functor $F$ in $(\B,\Ab)$ can be embedded into a functor of
  the form $(I\otimes_R-)|_\B$ where $I$ has cosupport in $\B$.
\end{lem}

\begin{proof}
  Applying Gabriel \cite[(proof of) II.\S1 prop.~3]{PG62} to
  \mbox{$\Hom_\ZZ(F,\QZ)$} we get a family of index sets $\{U_B\}_{B
    \in \B}$ and an exact sequence in $(\Bop,\Ab)$ of the form,
  \begin{displaymath}
    \textstyle{\coprod}_{B \in \B} \B(-,B)^{(U_B)}
    \lora \Hom_\ZZ(F,\QZ) \lora 0.
  \end{displaymath}
  Applying $\Hom_\ZZ(-,\QZ)$ to this sequence, we get an exact
  sequence in $(\B,\Ab)$,
  \begin{displaymath}
    0 \lora \Hom_\ZZ(\Hom_\ZZ(F,\QZ),\QZ) \lora 
    \Hom_\ZZ\big(\textstyle{\coprod}_{B \in \B} \B(-,B)^{(U_B)},\QZ\big).
  \end{displaymath}
  The module $I$ defined by \mbox{$\textstyle{\prod}_{B \in
      \B}\Hom_\ZZ(B,\QZ)^{U_B}$} has cosupport in $\B$, and we have a
  natural equi\-va\-lence for the latter functor above:
  \begin{align*}
    \Hom_\ZZ\big(\textstyle{\coprod}_{B \in \B}
    \B(-,B)^{(U_B)},\QZ\big) 
    &\,\simeq\,
    \textstyle{\prod}_{B \in \B}\Hom_\ZZ(\B(-,B),\QZ)^{U_B} \\ 
    &\,\simeq\,
    \textstyle{\prod}_{B \in
      \B}(\Hom_\ZZ(B,\QZ)\otimes_R-)|_\B^{\phantom{i}U_B} \\  
    &\,\simeq\,
    (I \otimes_R-)|_\B.
  \end{align*}
  The second $\simeq$ follows by Proposition \eqref{prp:formula}(b),
  and the third one since the tensor embedding commutes with products.
  To finish the proof, we need only note that $F$ embeds into its
  double Pontryagin dual $\Hom_\ZZ(\Hom_\ZZ(F,\QZ),\QZ)$.
\end{proof}

\begin{lem}
  \label{lem:summand}
  If a functor $F$ in $(\B,\Ab)$ is a direct summand of
  $(I\otimes_R-)|_\B$ where $I$ has cosupport in $\B$, then $F(-)
  \simeq (J\otimes_R-)|_\B$ for some $J$ with cosupport in $\B$.
\end{lem}

\begin{proof}
  Assume that we have a decomposition of $(I\otimes_R-)|_\B$ in
  $(\B,\Ab)$, say,
  \begin{displaymath}
    G(-) \oplus F(-) \simeq (I\otimes_R-)|_\B.
  \end{displaymath}
  Since finite products and coproducts agree in $(\B,\Ab)$, and since
  the tensor embedding commutes with products by Proposition
  \eqref{prp:formula}(a), we can use Eilenberg's swindle to obtain a
  the following natural equivalence of functors $\B \lora \Ab$,
  \begin{align*}
    F \oplus (I^{\mathbb{N}}\otimes_R-)|_\B &\,\simeq\,
    F \times (I\otimes_R-)|_\B^{\,\,\mathbb{N}} \\
    &\,\simeq\, F \times (I\otimes_R-)|_\B \times (I\otimes_R-)|_\B
    \times \cdots \\
    &\,\simeq\,
    F \times (G \times F) \times (G \times F) \times \cdots \\
    &\,\simeq\,
    (F \times G) \times (F \times G) \times (F \times \cdots \\
    &\,\simeq\, (I\otimes_R-)|_\B \times (I\otimes_R-)|_\B \times
    \cdots \\
    &\,\simeq\, (I^{\mathbb{N}}\otimes_R-)|_\B.
  \end{align*}
  In particular, we have a short exact sequence in $(\B,\Ab)$ given by
  \begin{displaymath}
    \tag{\text{$1$}}    
    0 \lora (I^\mathbb{N}\otimes_R-)|_\B \stackrel{u}{\lora}
    (I^\mathbb{N}\otimes_R-)|_\B \lora F(-) \lora 0.
  \end{displaymath}
  Since $I$ is in $(\ModR)^\B$, then so is the product $I^\mathbb{N}$.
  Thus, it follows by Proposition \eqref{prp:formula}(d) that $u$ is
  induced by a module homomorphism \mbox{$h \colon I^\mathbb{N} \lora
    I^\mathbb{N}$}, and exactness of $(1)$ shows that $h$ is a
  $\B$-monomorphism.  Now, Lemma \eqref{lem:Bsplit} implies that $h$
  is a split mo\-no\-morphism, so defining
  \mbox{$J=\operatorname{Coker}h$} gives a split exact sequence in
  $\ModR$,
  \begin{displaymath}
    \tag{\text{$2$}}    
    0 \lora I^\mathbb{N} \stackrel{h}{\lora}
    I^\mathbb{N} \lora J \lora 0.
  \end{displaymath}
  As a direct summand of $I^\mathbb{N}$, the module $J$ has cosupport
  in $\B$.  The sequence $(2)$ induces a (split) exact sequence in
  $(\B,\Ab)$,
  \begin{displaymath}
    \tag{\text{$3$}}    
    0 \lora (I^\mathbb{N}\otimes_R-)|_\B \stackrel{u}{\lora}
    (I^\mathbb{N}\otimes_R-)|_\B \lora (J\otimes_R-)|_\B \lora 0,
  \end{displaymath}
  and the desired conclusion follows by comparing $(1)$ and $(3)$.
\end{proof}

We are now ready to prove Theorem A from the Introduction. Note that
this results is well-known in the case where
\mbox{$\B=R\text{-}\mod$}, see for example
\cite[thm.\,B.16]{CUJ-HL-book}.


\begin{proof}[\textbf{Proof of Theorem A}]
  First we must argue that the functor \mbox{$(I\otimes_R-)|_\B$} is
  injective if $I$ has cosupport in $\B$. By Definition
  \eqref{dfn:cosupport} and Proposition \eqref{prp:formula}(a) we may
  assume that $I$ has the form $\Hom_\ZZ(B,\QZ)$ for some $B$ in $\B$.
  Now, let
  \begin{displaymath}
    \Xi \,=\, \ 0 \lora F' \lora F \lora F'' \lora 0
  \end{displaymath}
  be a short exact sequence in $(\B,\Ab)$, in particular,
  \begin{displaymath}
    \tag{\text{$1$}}
    0 \lora \Hom_\ZZ(F''B,\QZ) \lora \Hom_\ZZ(FB,\QZ) \lora
    \Hom_\ZZ(F'B,\QZ) \lora 0 
  \end{displaymath}
  is exact in $\Ab$. By Proposition \eqref{prp:formula}(c), the
  sequence $(1)$ is isomorphic to
  \begin{displaymath}
    \tag{\text{$2$}}
    (\B,\Ab)\big(\Xi,(\Hom_\ZZ(B,\QZ)\otimes_R-)|_\B\big)=    
    (\B,\Ab)\big(\Xi,(I\otimes_R-)|_\B\big),
  \end{displaymath}
  and since $(1)$ is exact then so is $(2)$. Thus, $(I\otimes_R-)|_\B$
  is injective in $(\B,\Ab)$.

  To show that the tensor embedding gives the claimed
  equivalence, we argue that it is fully faithful and essentially
  surjective as a functor from $(\ModR)^\B$ to $\Inj(\B,\Ab)$.

  By Proposition \eqref{prp:formula}(d), the restriction of the tensor
  embedding to $(\ModR)^\B$ is fully faithful, and essential
  surjectivity follows from Lemmas \eqref{lem:embedding} and
  \eqref{lem:summand}.
\end{proof}

\begin{rmk}
  By Theorem A, every $F$ in $\Inj(\B,\Ab)$ has the form
  \mbox{$F\simeq(I \otimes_R-)|_\B$} for a unique (up to isomorphism)
  module $I$ with cosupport in $\B$.  However, if the functor
  \mbox{$(I \otimes_R-)|_\B$} is injective, $I$ need not have
  cosupport in $\B$.

  For example, if as in \eqref{exa:not_full_not_faithful} we let
  $R=\ZZ$ and $\B=\add\,\ZZ/(p)$, it follows from the isomorphism
  \mbox{$\Hom_\ZZ(\ZZ/(p),\QZ) \cong \ZZ/(p)$} that $\ZZ/(q)$ does not
  have cosupport in $\B$.  However, $(\ZZ/(q) \otimes_\ZZ-)|_\B$ is
  the zero functor and thus it is injective.
\end{rmk}

\begin{prp}
  \label{prp:cosupp_converse}
  Let $N$ be in $\ModR$. Then \mbox{$(N \otimes_R-)|_\B$} is in
  $\Inj(\B,\Ab)$ if and only if there exist $I$ in $(\ModR)^\B$ and a
  \mbox{$(-\otimes_R\B)$}-isomorphism $N \lora I$.

  Thus, if $R\in\varinjlim\B$ then $(N \otimes_R-)|_\B \in
  \Inj(\B,\Ab)$ if and only if $N \in (\ModR)^\B$.
\end{prp}

\begin{proof}
  The first assetion is clear from Proposition \eqref{prp:formula}(d)
  and Theorem A. For the last claim we note that if $\varphi$ is a
  \mbox{$(-\otimes_R\B)$}-isomorphism and $R \cong \varinjlim B_i$
  with $B_i \in \B$ then $\varphi$ is an isomorphism since $\varphi =
  \varphi \otimes_R R = \varphi \otimes_R (\varinjlim B_i) =
  \varinjlim(\varphi \otimes_R B_i)$.
\end{proof}

\begin{prp}
  \label{prp:supp_converse}
  Let $M$ be in $\RMod$. Then $\Hom_R(-,M)|_\B$ is in
  $\Flat(\Bop,\Ab)$ if and only if there exist $M$ in $\varinjlim\B$
  and a $\Hom_R(\B,-)$-isomorphism $F \lora M$.

  Thus, if $R\in\varinjlim\B$ then $\Hom_R(-,M)|_\B \in
  \Flat(\Bop,\Ab)$ if and only if $M \in \varinjlim\B$.
\end{prp}

\begin{proof}
  By the proof of \cite[prop.~2.4]{HL83}, the homomorphism of abelian
  groups,
  \begin{displaymath}
    \Hom_R(F,M) \lora (\Bop,\Ab)(\Hom_R(-,F)|_\B,\Hom_R(-,M)|_\B),
  \end{displaymath}
  induced by the Yoneda functor \mbox{$\RMod \lora (\Bop,\Ab)$} is an
  isomorphism for all $F$ in $\varinjlim\B$. From this fact and from
  \cite[prop.~2.4]{HL83} the first assertion follows.

  Finally, if $\varphi$ is a \mbox{$\Hom_R(\B,-)$}-isomorphism and $R
  \cong \varinjlim B_i$ with $B_i \in \B$ then $\varphi$ is an
  isomorphism since $\varphi = \Hom_R(R,\varphi) = \Hom_R(\varinjlim
  B_i,\varphi) = \varprojlim\Hom_R(B_i,\varphi)$.
\end{proof}

\enlargethispage{2ex}

\section{Covers and envelopes by modules with (co)support}
\label{sec:EnvCov}

The reader is assumed to be familiar with the notions of precovering
(contravariantly finite), preenveloping (covariantly finite),
covering, and enveloping subcategories. We refer to
e.g.~\cite[chap.~5.1 and 6.1]{EEE-OMGJ-book} for the relevant
definitions.

\bigskip

By El~Bashir~\cite[thm.~3.2]{REB06} the class $\varinjlim\B$ is
covering, in particular, it is closed under coproducts in $\RMod$.
The next result due to Crawley-Boevey \cite[thm.~(4.2)]{WCB94}
and Krause \cite[prop.~3.11]{HK01} characterizes when $\varinjlim\B$
is closed under products.

\begin{thm}
  \label{thm:Bcoherent}
  The following conditions are equivalent:
  \begin{eqc}
  \item $\varinjlim\B$ is closed under products in $\RMod$;
  \item $\varinjlim\B$ is preenveloping in $\RMod$;
  \item $\B$ is preenveloping in $R\text{-}\mod$;
  \item $\varinjlim\B$ is definable.
  \end{eqc}
\end{thm}

\begin{dfn}
  \label{dfn:Bcoherent}
  $R$ is called \textsl{$\B$-coherent} if the conditions in
  \eqref{thm:Bcoherent} are satisfied.
\end{dfn}

\begin{exa}
  \label{exa:coherent}
  The following conclusions hold.
  \begin{prt}
  \item If \mbox{$\B=R\text{-}\proj$} then
    \mbox{$\varinjlim\B=R\text{-}\Flat$} by Lazard \cite{DL64}, so by
    Chase \cite[thm.~2.1]{SUC60}, $R$ is $\B$-coherent if and only
    if it is right coherent in the classical sense.
  \item If \mbox{$\B=R\text{-}\mod$} then \mbox{$\varinjlim\B=\RMod$}
    by \cite[(7.15)]{CUJ-HL-book}, so all rings are
    $\B$-coherent.
  \end{prt}
\end{exa}

As an easy application of Theorem A, we now prove Theorem B.  
In view of Example \eqref{exa:ordinary_inj}, Theorem B implies the
existence of injective hulls and pure injective envelopes.  The first
of these classical results was proved by Eckmann and Schopf
\cite{BE-AS-53}, and the second one by Fuchs \cite{LF67} and
Kie{\l}pi{\'n}ski \cite{RK67}.

\begin{dfn}
  \label{dfn:essential}
  A homomorphism \mbox{$h \colon N \lora M$} of right $R$-modules is
  called a \textsl{$\B$-essential $\B$-monomorphism} if it is a
  $\B$-monomorphism in the sense of Definition \eqref{dfn:Bmono} and
  if any homomorphism \mbox{$g \colon M \lora L$} is a
  $\B$-monomorphism if \mbox{$g \circ h$} is so.
\end{dfn}


\begin{proof}[\textbf{Proof of Theorem B}]
  Since an envelope is unique up to isomorphism,
  cf.~\cite[prop.~1.2.1]{JX96}, it suffices to argue that every $N$
  admits a $\B$-essential $\B$-mono\-mor\-phism \mbox{$h \colon N
    \lora I$} with $I$ in $(\ModR)^\B$, and that every such map is a
  $(\ModR)^\B$-envelope.

  To this end, let \mbox{$u \colon (N\otimes_R-)|_\B \lora U$} be an
  injective hull in $(\B,\Ab)$, see \eqref{bfhpg:functor_categories}.
  By The\-o\-rem A, the functor $U$ has the form
  \mbox{$(I\otimes_R-)|_\B$} for an $I$ with cosupport in $\B$, and by
  Proposition \eqref{prp:formula}(d), $u$ is induced by a homomorphism
  \mbox{$h \colon N \lora I$}. It is easily seen that $h$ is a
  $\B$-essential $\B$-monomorphism.  Another application of Theorem A,
  combined with \cite[II.\S5, prop.~8]{PG62}, gives that if $h$ is a
  $\B$-essential $\B$-monomorphism then it is also a
  $(\ModR)^\B$-envelope.
\end{proof}

We are now ready to prove Theorem C, which characterizes when
$(\ModR)^\B$ is closed under coproducts. The hard parts of the proof
of Theorem C follow from references to works of Angeleri-H{\"u}gel,
Krause, Rada, and Saor{\'{\i}}n, \cite{LAH00}, \cite{HK01},
\cite{HK-MS-98}, \cite{JR-MS-98}.


\begin{proof}[\textbf{Proof of Theorem C}]
  It suffices to prove the implications:
  \begin{displaymath}
    \xymatrix@R=0ex@C=3ex{
    (iii) \ar@{=>}[r] & (i) \ar@{=>}[dd] & (ii) \ar@{=>}[l] & {} \\
    {} & {} & {} & (vi) \ar@{=>}[ul] \\
    (iv) \ar@{=>}[uu] & (vii) \ar@{=>}[l] \ar@{=>}[r] & (v) \ar@{=>}[ur]
    & {} \\
    }
  \end{displaymath}

  ``\mbox{$(i) \Rightarrow (vii)$}'': Note that
  \mbox{$(\ModR)^\B=\Prod\,J$}, where $J$ is $\prod_{\alpha \in
    A}\Hom_\ZZ(B_\alpha,\QZ)$ and $\{B_\alpha\}_{\alpha \in A}$ is a
  set of re\-pre\-sen\-tatives for the isomorphism classes in $\B$.
  By Definition \eqref{dfn:cosupport}, all modules in $(\ModR)^\B$ are
  pure injective. Hence $(i)$ implies that $J$ is
  $\Sigma$-pure-injective, and the proof of \cite[prop.~6.10]{LAH00} gives the desired
  conclusion.

  ``\mbox{$(vii) \Rightarrow (iv)$}'': If $(vii)$ holds then $E$ is
  product complete, cf.~\cite[\S3]{HK-MS-98}, and it follows by
  \cite[cor.~3.6]{HK-MS-98} that $(\ModR)^\B$ is closed under direct
  limits. By \cite[cor.~3.7(a)]{JR-MS-98}, the class $(\ModR)^\B$ is
  also precovering, and hence it is covering by
  \cite[thm.~2.2.8]{JX96}.
  
  ``\mbox{$(iv) \Rightarrow (iii) \Rightarrow (i)$}'': The first
  implicaton is trivial, and the latter is a consequence of
  \cite[thm.~3.4]{JR-MS-98} since $(\ModR)^\B$ is closed under direct
  summands.

  ``\mbox{$(vii) \Rightarrow (v)$}'': By \cite[thm.~6.7]{HK01}, the
  assumption $(vii)$ ensures that $(\ModR)^\B$ is definable, in
  particular, it is closed under pure submodules,
  cf.~\cite[thm.~2.1]{HK01}. 

  ``\mbox{$(v) \Rightarrow (vi)$}'': Let \mbox{$\eta =\,0 \to N' \to N
    \to N'' \to 0$} be pure exact. If $N'$ and $N''$ are in
  $(\ModR)^\B$ then, as $N'$ is pure injective, $\eta$ splits and
  \mbox{$N \cong N' \oplus N'' \in (\ModR)^\B$}. If $N$ is in
  $(\ModR)^\B$, the assumption $(v)$ gives that $N'$ is in
  $(\ModR)^\B$. As before, the sequence splits, and $N''$ is in
  $(\ModR)^\B$ since it is a direct summand of $N$.

  ``\mbox{$(vi) \Rightarrow (ii)$}'': Let \mbox{$\varphi_{\mu\lambda}
    \colon I_\lambda \lora I_\mu$} be a direct system of modules
  from $(\ModR)^\B$. As $\prod E_\lambda$ is in $(\ModR)^\B$,
  as \mbox{$\coprod E_\lambda \lora \prod E_\lambda$} is a pure
  monomorphism, and since \mbox{$\coprod E_\lambda \lora \varinjlim
    E_\lambda$} is a pure epimorphism, we conclude that $\varinjlim
  E_\lambda$ is in $(\ModR)^\B$.

  ``\mbox{$(ii) \Rightarrow (i)$}'': A coproduct is the direct limit
  of its finite sub-coproducts.
\end{proof}

\begin{dfn}
  \label{dfn:Bnoetherian}
  $R$ is \textsl{$\B$-noetherian} if the conditions in Theorem
  C are satisfied.
\end{dfn}

\begin{exa}
  The following conclusions hold.
  \begin{prt}
  \item If \mbox{$\B=R\text{-}\proj$} then
    \mbox{$(\ModR)^\B=\Inj\text{-}R$} cf.~\eqref{exa:ordinary_inj}(a),
    so by Bass \cite[thm.~1.1]{HB62}, $R$ is $\B$-noetherian if
    and only if it is right noetherian in the classical sense.
  \item If \mbox{$\B=R\text{-}\mod$} then
    \mbox{$(\ModR)^\B=\PureInj\text{-}R$}
    cf.~\eqref{exa:ordinary_inj}(b), so by
    \cite[thm.~B.18]{CUJ-HL-book}, $R$ is $\B$-noetherian if and only
    if it is right pure semi-simple.
  \end{prt}
\end{exa}

\begin{cor}
  Assume that \mbox{$\B \subseteq \B'$} are two additive full
  subcategories of $R\text{-}\mod$. If the ring $R$ is
  $\B'$-noetherian then it is also $\B$-noetherian.
\end{cor}

\begin{proof}
  Assume that $R$ is $\B'$-noetherian and let $\{E_\lambda\}$ be a
  family in $(\ModR)^\B$. By our asumptions, \mbox{$(\ModR)^\B
    \subseteq (\ModR)^{\B'}$}, and the latter is closed under
  coproducts. It follows that $\coprod E_\lambda$ belongs to
  $(\ModR)^{\B'}$ and, in particular, $\coprod E_\lambda$ is pure
  injective. Thus, the pure monomorphism \mbox{$\coprod E_\lambda
    \lora \prod E_\lambda$} is split, and since $\prod E_\lambda$
  belongs to $(\ModR)^\B$ then so does $\coprod E_\lambda$. Thus $R$
  is $\B$-noetherian by Theorem C.
\end{proof}

It is natural to ask if there exists a cotorsion pair
$(\M,(\ModR)^\B)$ of finite type?

\begin{exa}
  The following is well-known.
  \begin{prt}
  \item If \mbox{$\B=R\text{-}\proj$} then
    \mbox{$(\ModR)^\B=\Inj\text{-}R$},
    cf.~\eqref{exa:ordinary_inj}(a). Clearly $(\ModR,\Inj\text{-}R)$
    is a cotorsion pair which is of finite type if $R$ is right
    noetherian.
  \item If \mbox{$\B=R\text{-}\mod$} then
    \mbox{$(\ModR)^\B=\PureInj\text{-}R$},
    cf.~\eqref{exa:ordinary_inj}(b). In general, there does not exist
    a cotorsion pair of the form $(\M,\PureInj\text{-}R)$.
  \end{prt}
\end{exa}

Our proof of the following result uses Theorems D and E which are
proved in the next section. However, Proposition \eqref{prp:cotorsion}
itself naturally belongs in this section.

\begin{prp}
  \label{prp:cotorsion}
  Assume that $R$ is right coherent. Then there exists a cotorsion
  pair of finite type $(\M,(\ModR)^\B)$ if and only if $\B$ satisfies
  the following conditions:
  \begin{rqm}
  \item $R$ belongs to $\varinjlim\B$;
  \item $R$ is $\B$-noetherian;
  \item If $\,\Tor_1^R(M,F)=0$ for all $M \in \mod\text{-}R$ with
    $\Tor_1^R(M,\B)=0$ then $F \in \varinjlim\B$.
  \end{rqm}
\end{prp}

\begin{proof}
  ``If'': First assume that (1)--(3) hold. By the isomorphism
  \cite[VI.\S5]{CE56},
  \begin{displaymath}
    \tag{\text{$\dagger$}}
    \Ext^1_{\Rop}(-,\Hom_\ZZ(B,\QZ)) \simeq \Hom_\ZZ(\Tor_1^R(-,B),\QZ),
  \end{displaymath}
  it follows that \mbox{$\M := \Ker\Ext_{\Rop}^1(-,(\ModR)^\B)=
    \Ker\Tor_1^R(-,\B)$}.  To prove that $(\M,(\ModR)^\B)$ is a
  cotorsion pair of finite type, we show that $E$ is in $(\ModR)^\B$
  if \mbox{$\Ext_{\Rop}^1(\M \cap \mod\text{-}R,E)=0$}. As $R$ is
  right coherent, \cite[prop.~VI.5.3]{CE56} gives that
  \begin{displaymath}
    \tag{\text{$\ddagger$}}
    \Tor_1^R(M,\Hom_\ZZ(E,\QZ)) \cong \Hom_\ZZ(\Ext_{\Rop}^1(M,E),\QZ)
  \end{displaymath}
  for $M$ in $\mod\text{-}R$. In light of $(\ddagger)$ and (3),
  \mbox{$\Ext_{\Rop}^1(\M \cap \mod\text{-}R,E)=0$} implies that
  $\Hom_\ZZ(E,\QZ)$ is in $\varinjlim\B$; so (1), (2), and Theorem D
  gives that $E$ is in $(\ModR)^\B$.

  ``Only if'': Assume that $(\M,(\ModR)^\B)$ is a cotorsion pair of
  finite type. Since $\Hom_\ZZ(R,\QZ)$ is injective, it belongs to
  $(\ModR)^\B$, and thus Theorem D gives (1). If
  $(\mathcal{F},\mathcal{G})$ is any cotorsion pair of finite type in
  $\ModR$ over a right coherent ring then $\mathcal{G}$ is closed
  under coproducts; this proves (2). We have \mbox{$\M =
    \Ker\Tor_1^R(-,\B)$} by $(\dagger)$, and hence
  \cite[thm.~2.3]{LAH-JT-04} implies that \mbox{$\varinjlim(\M \cap
    \mod\text{-}R)=\M$}. To prove (3) we assume that
  \mbox{$\Tor_1^R(\M \cap \mod\text{-}R,F)=0$}. By the preceding
  remark, it follows that \mbox{$\Tor_1^R(\M,F)=0$}, and combining
  this with the isomorphism,
  \begin{displaymath}
    \Ext^1_{\Rop}(M,\Hom_\ZZ(F,\QZ)) \cong \Hom_\ZZ(\Tor_1^R(M,F),\QZ),
  \end{displaymath}
  and the fact that $(\M,(\ModR)^\B)$ is a cotorsion pair, we conclude
  that $\Hom_\ZZ(F,\QZ)$ is in $(\ModR)^\B$. By Theorem D it follows
  that $F$ is in $\varinjlim\B$.
\end{proof}

\section{Stability results} \label{sec:stability}

In this section we prove a number of stability results for modules
with (co)support in $\B$, and we also present some applications. The
terminology in Definitions \eqref{dfn:Bcoherent} and
\eqref{dfn:Bnoetherian} play a central role in this section.

\begin{bfhpg}[Injective structures]
  \label{bfhpg:injective_structures}
  Maranda \cite{JMM64} defines an \textsl{injective structure} as a
  pair $(\mathcal{H},\mathcal{Q})$ where $\mathcal{H}$ is a class of
  homomorphisms and $\mathcal{Q}$ is a class of modules satisfying:
  \begin{rqm}
  \item $Q \in \mathcal{Q}$ if and only if $\Hom_R(h,Q)$ is surjective
    for all $h \in \mathcal{H}$;
  \item $h \in \mathcal{H}$ if and only if $\Hom_R(h,Q)$ is surjective
    for all $Q \in \mathcal{Q}$;
  \item For every $R$-module $M$ there exists $h \colon M \lora Q$
    where $h \in \mathcal{H}$ and $Q \in \mathcal{Q}$.
  \end{rqm}
  Given (2), condition (3) means exactly that $\mathcal{Q}$ is
  preenveloping in $\ModR$.

  Enochs-Jenda-Xu \cite[thm.~2.1]{EEE-OMGJ-JX-93} prove that if
  $\mathcal{H}$ is the class of $\B$-monomorphisms, cf.~Definition
  \eqref{dfn:Bmono}, then $(\mathcal{H},(\ModR)^\B)$ is an injective
  structure, and $(\ModR)^\B$ is enveloping (not just preenveoping).
  The last fact also follows from Theorem B.
\end{bfhpg}

\begin{lem}
  \label{lem:exact-exact}
  Let $\xi$ be a complex of left $R$-modules, and let $\eta$ be a
  complex of right $R$-modules. Then the following conclusions hold:
  \begin{prt}
  \item $\xi$ is $\Hom_R(\B,-)$-exact if and only if
    $\,\Hom_{\ZZ}(\xi,\QZ)$ is $(-\otimes_R\B)$-exact.
  \item $\eta$ is $(-\otimes_R\B)$-exact if and only if
    $\,\Hom_{\ZZ}(\eta,\QZ)$ is $\Hom_R(\B,-)$-exact.
  \end{prt}
\end{lem}

\begin{proof}
  For a finitely presented left $R$-module $B$, there are natural
  isomorphisms,
  \begin{align*}
    \Hom_{\ZZ}(\Hom_R(B,\xi),\QZ)  &\cong \Hom_{\ZZ}(\xi,\QZ)\otimes_RB, 
    \\
    \Hom_{\ZZ}(\eta\otimes_RB,\QZ) &\cong \Hom_R(B,\Hom_{\ZZ}(\eta,\QZ)),
  \end{align*}
  see \cite[prop.~VI.5.3]{CE56} and \cite[prop.~II.5.2]{CE56}.
  From these the lemma easily follows.
\end{proof}


\enlargethispage{2ex}

\begin{proof}[\textbf{Proof of Theorem D}]
  ``Only if'': If \mbox{$F \in \varinjlim\B$} and $h$ is a
  $\B$-monomorphism, then \mbox{$h \otimes_RF$} is injective since
  $\otimes$ commutes with $\varinjlim$ by
  \cite[cor.~2.6.17]{CAW-book}.  Thus
  \begin{displaymath}
    \Hom_{\Rop}(h,\Hom_{\ZZ}(F,\QZ)) \cong
    \Hom_{\ZZ}(h\otimes_RF,\QZ) 
  \end{displaymath}
  is surjective, and it follows from
  \eqref{bfhpg:injective_structures} that $\Hom_{\ZZ}(F,\QZ)$ belongs
  to $(\ModR)^\B$.

  ``If'': By \cite[cor.~3.7(a)]{JR-MS-98} the class of modules
  consisting of coproducts of modules from $\B$ is precovering. Hence
  there is a left-exact and $\Hom_R(\B,-)$-exact sequence,
  \begin{displaymath}
    \xi \, = \, \ 0 \lora K \lora P \stackrel{\pi}{\lora} F \lora 0,
  \end{displaymath}
  where $P$ is a set-indexed coproduct of modules from $\B$.  A priori
  we do not know if $\xi$ is exact at $F$, but we will argue that
  $\xi$ is, in fact, pure exact. Having showed this, it will follow
  from \cite[prop.~2.1]{HL83} that $F$ belongs to $\varinjlim\B$, as
  desired.

  Exactness and pure exactness of $\xi$ can be proved simultaneously
  by showing that $\Hom_\ZZ(\pi,\QZ)$ is a split monomorphism,
  cf.~\cite[thm.~6.4]{CUJ-HL-book}. As $\Hom_{\ZZ}(F,\QZ)$ is in
  $(\ModR)^\B$, it suffices by Lemma \eqref{lem:Bsplit} to see that
  $\Hom_\ZZ(\pi,\QZ)$ is a $\B$-monomor\-phism, but this follows from
  Lemma \eqref{lem:exact-exact}(a) and $\Hom_R(\B,-)$-surjectivity of
  $\pi$.
\end{proof}

\begin{cor}
  \label{cor:cap}
  Assume that $\E$ is a class of right $R$-modules that is closed
  under direct summands and products in $\ModR$ and satisfies the
  conditions:
  \begin{rqm}
  \item $\Hom_\ZZ(B,\QZ)$ belongs to $\E$ for every $B \in \B$;
  \item $\Hom_\ZZ(E,\QZ)$ belongs to $\varinjlim\B$ for every $E \in
    \E$.
  \end{rqm}
  Then there is an equality, $(\ModR)^\B = \E \cap \PureInj\text{-}R$.
\end{cor}

\begin{proof}
  The inclusion ``$\subseteq$'' is clear from (1).  To prove
  ``$\supseteq$'' we assume that \mbox{$E \in \E$} is pure injective.
  As $E$ is in $\E$, it follows by (2) and Theorem D that the module
  $D(E)$ defined by $\Hom_\ZZ(\Hom_\ZZ(E,\QZ),\QZ)$ belongs to
  $(\ModR)^\B$.  As the canonical homomorphism \mbox{$E \lora D(E)$}
  is a pure monomorphism, and since $E$ is pure injective, $E$ is a
  direct summand of $D(E)$. Consequently, $E$ belongs to $(\ModR)^\B$.
\end{proof}

Applying \eqref{lem:Bsplit}, \eqref{bfhpg:injective_structures},
\eqref{lem:exact-exact}, and Theorem D, it is easy to prove the
following properties for modules with support in $\B$, akin to those
found in Lenzing \cite[\S2]{HL83}.

\begin{cor}
  \label{cor:Lenzing-like}
  The following conclusions hold:
  \begin{prt}
  \item A left $R$-module $F$ belongs to $\varinjlim\B$ if and only if
    $h \otimes_RF$ is a monomorphism for every $\B$-monomorphism $h$.
  \item If\, \mbox{$0 \to F' \to F \to F'' \to 0$} is an exact and
    $\,\Hom_R(\B,-)$-exact sequence with $F''$ in $\varinjlim\B$, then
    $F'$ is in $\varinjlim\B$ if and only if $F$ is in $\varinjlim\B$.
    \qed
  \end{prt}
\end{cor}

\begin{obs}
  \label{obs:coh-coh-noet-noet}
  Note that if $R\in\varinjlim\B$, Proposition
  \eqref{prp:supp_converse}/\eqref{prp:cosupp_converse} implies that
  the ring $R$ is $\B$-coherent/-noetherian in the sense of Definition
  \eqref{dfn:Bcoherent}/\eqref{dfn:Bnoetherian} if and only if the
  category $\B$ is left coherent/noetherian in the sense of Definition
  \eqref{dfn:cat_noeth}.
\end{obs}

\begin{thm}
  \label{thm:B-inj_B-flat}
  Assume that $R$ is in $\varinjlim\B$ and that $R$ is $\B$-coherent.
  Then a right $R$-module $E$ is in $(\ModR)^\B$ only if
  $\,\Hom_{\ZZ}(E,\QZ)$ is in $\varinjlim\B$.
\end{thm}

\begin{proof}
  We have the following implications,
  \begin{align*}
    E \text{ is in } (\ModR)^\B 
     &\ \Longleftrightarrow \ 
   (E \otimes_R-)|_\B \text{ is in } \Inj(\B,\Ab) \\
     &\ \,\Longrightarrow \ 
   \Hom_\ZZ((E \otimes_R-)|_\B,\QZ) \text{ is in } \Flat(\Bop,\Ab) \\
     &\ \Longleftrightarrow \ 
   \Hom_R(-,\Hom_\ZZ(E,\QZ))|_\B \text{ is in } \Flat(\Bop,\Ab) \\
     &\ \Longleftrightarrow \ 
   \Hom_\ZZ(E,\QZ) \text{ is in } \varinjlim\B.
  \end{align*}
  The first and last equivalences follow from Propositions
  \eqref{prp:cosupp_converse} and \eqref{prp:supp_converse}, and the
  penultimate equivalence is by adjunction. The implication in the
  second line is immediate by Observation
  \eqref{obs:coh-coh-noet-noet} and Proposition
  \eqref{prp:dual_to_injective}.
\end{proof}

\enlargethispage{7ex}

A result by Gruson and Jensen \cite{LG-CUJ-80} and Enochs
\cite[lem.~1.1]{EEE87} asserts that over a right coherent ring, the
pure injective envelope of a flat left $R$-module is again flat.  In
view of Example \eqref{exa:coherent}(b), we have the following
generalization.

\begin{cor}
  \label{cor:PE}
  Assume that \mbox{$R\in\varinjlim\B$} and that $R$ is $\B$-coherent.
  If $F$ is in $\varinjlim\B$ then its pure injective envelope $PE(F)$
  and the quotient $PE(F)/F$ are in $\varinjlim\B$.
\end{cor}

\begin{proof}
  We have a pure monomorphism \mbox{$F \lora
    D(F)=\Hom_\ZZ(\Hom_\ZZ(F,\QZ),\QZ)$}, and since $D(F)$ is pure
  injective it contains $PE(F)$ as a direct summand. Theorems D and
  \eqref{thm:B-inj_B-flat} imply that $D(F)$ is in $\varinjlim\B$, and
  we conclude that $PE(F)$ is in $\varinjlim\B$.  That $PE(F)/F$ is in
  $\varinjlim\B$ now follows from \cite[prop.~2.2]{HL83}.
\end{proof}


\begin{proof}[\textbf{Proof of Theorem E}]
  In view of the proof of Theorem \eqref{thm:B-inj_B-flat}, Theorem E
  is an immediate consequence of Observation
  \eqref{obs:coh-coh-noet-noet} and Proposition \eqref{prp:Bnoet-eqc}.
\end{proof}

\appendix

\section{Two results on flat and injective functors} \label{sec:app}

Propositions \eqref{prp:dual_to_injective} and \eqref{prp:Bnoet-eqc}
below play a central role in the proofs of Theorems
\eqref{thm:B-inj_B-flat} and E.  Since we have not been able to find
proofs of \eqref{prp:dual_to_injective} or \eqref{prp:Bnoet-eqc} in
the literature, they are included in this appendix.

\begin{lem}
  \label{lem:homeval}
  For $F, G \in (\B,\Ab)$ and $A \in \Ab$ there is a canonical
  homomorphism,
  \begin{displaymath}
    \xymatrix@C=8ex{
    \Hom_\ZZ(G,A) \otimes_\B F \ar[r]^-{\omega_{GAF}} &
    \Hom_\ZZ((\B,\Ab)(F,G),A).
    }
  \end{displaymath}
  If $A$ is injective (divisible) and $F$ is finitely presented then
  $\omega_{GAF}$ is an isomorphism.
\end{lem}

\begin{proof}
  For each $B$ in $\B$ there is canonical homomorphism of abelian
  groups,
  \begin{displaymath} 
    \xymatrix{
    \Hom_\ZZ(GB,A) \ar[r]^-{\varpi_B} &
    \Hom_\ZZ(FB,\Hom_\ZZ((\B,\Ab)(F,G),A)),
    }
  \end{displaymath}
  It is given by \mbox{$\varpi_B(f)(x)(\theta)=(f \circ \theta_B)(x)$}
  where \mbox{$f \colon GB \lora A$} is a homomorphism, \mbox{$x \in
    FB$} is an element, and \mbox{$\theta \colon F \lora G$} is a
  natural transformation. It is easily seen that $\varpi$ is a natural
  transformation of functors $\Bop \lora \Ab$. By applying
  \eqref{bfhpg:flat_functors}(a),
  \begin{align*} 
    \varpi &\in (\Bop,\Ab)\big( \Hom_\ZZ(G,A),
    \Hom_\ZZ(F,\Hom_\ZZ((\B,\Ab)(F,G),A)) \big) \\
    &\cong \Hom_\ZZ\!\big(
    \Hom_\ZZ(G,A) \otimes_\B F,
    \Hom_\ZZ((\B,\Ab)(F,G),A)\big), 
  \end{align*}
  it follows that $\varpi$ corresponds to the homomorphism which is
  denoted by $\omega$ in the lemma. It is straightforward to verify
  that $\omega$ is natural in $F$, $G$, and $A$.
  
  To see that $\omega_{GAF}$ is an isomorphism when $A$ is injective
  and $F$ is finitely presented, note that $\omega_{G?A}$ is a natural
  transformation between right exact and additive functors $(\B,\Ab)
  \lora \Ab$. As every finitely presented $F$ fits into an exact
  sequence,
  \begin{displaymath}
    \B(B_1,-) \lora \B(B_0,-) \lora F(-) \lora 0,
  \end{displaymath}
  it suffices, by the five-lemma, to check that $\omega_{G,\B(B,-),A}$
  is an isomorphism. However, this homomorphism is the composition if
  the following two isomorphismsm:
  \begin{displaymath}
    \Hom_\ZZ(G,A) \otimes_\B \B(B,-) \stackrel{\cong}{\lora}
    \Hom_\ZZ(GB,A) \stackrel{\cong}{\lora}
    \Hom_\ZZ((\B,\Ab)(\B(B,-),G),A),
  \end{displaymath}
  where the left-hand isomorphism is by
  \eqref{bfhpg:flat_functors}(b), and the right-hand isomorphism is by
  Yoneda's Lemma, cf.~\cite[III.\S2]{SML-book}. This finishes the
  proof.
\end{proof}

\begin{rmk}
  \label{rmk:Baer}
  By Oberst-R{\"o}hrl \cite[(proof of) thm.~(3.2)]{OR70} a functor
  $T$ in $(\Bop,\Ab)$ is flat if for every finitely generated additive
  subfunctor $F$ of a representable functor $\B(B,-)$, one has
  exactness of the sequence:
  \begin{displaymath}
    0 \lora T\otimes_\B F \lora T\otimes_\B\B(B,-).
  \end{displaymath}
  Although the author was not able to find a reference, it is
  well-known that Baer's Criterion holds in functor
  categories\footnote{\ One way to prove this is by combining the
    proof of Anderson-Fuller \cite[prop.~16.13]{FWA-KRF-92} with the
    first line in the proof of Lemma \eqref{lem:embedding}.}, that is,
  $E$ in $(\B,\Ab)$ is injective if for every additive subfunctor $G$
  of $\B(B,-)$, one has exactness of the sequence:
  \begin{displaymath}
    (\B,\Ab)(\B(B,-),E) \lora (\B,\Ab)(G,E) \lora 0.
  \end{displaymath}
\end{rmk}

\begin{dfn}
  \label{dfn:cat_noeth}
  The category $\B$ is \textsl{left coherent} if $\Flat(\Bop,\Ab)$ is
  closed under products in $(\Bop,\Ab)$; and $\B$ is \textsl{right
    coherent} if $\Bop$ is left coherent.

  The category $\B$ is \textsl{left noetherian} if $\Inj(\B,\Ab)$ is
  closed under coproducts in $(\B,\Ab)$; and $\B$ is \textsl{right
    noetherian} if $\Bop$ is left noetherian.
\end{dfn}

\begin{rmk}
  \label{rmk:coh-noet}
  Jensen-Lenzing \cite[thm.~B.17]{CUJ-HL-book} and Oberst-R{\"o}hrl
  \cite[thm.~(4.1)]{OR70} contain several equivalent characterizations
  of the notions above. For example, $\B$ if left coherent if and only
  if every \textsl{finitely generated} additive subfunctor of
  $\B(B,-)$ is finitely presented; and $\B$ if left noetherian if and
  only if \textsl{every} additive subfunctor of $\B(B,-)$ is finitely
  generated (and thus, finitely presented).
\end{rmk}

\begin{prp}
  \label{prp:dual_to_injective}
  Assume that $\B$ is left coherent. Then $E$ belongs to
  $\Inj(\B,\Ab)$ only if $\,\Hom_\ZZ(E,\QZ)$ belongs to
  $\Flat(\Bop,\Ab)$.
\end{prp}

\begin{proof}
  Let $F$ be a finitely generated additive subfunctor of $\B(B,-)$.
  As $\B$ is left coherent, $F$ is finitely presented by Remark
  \eqref{rmk:coh-noet}, and consequently Lemma \eqref{lem:homeval}
  gives the vertical isomorphisms in the following commutative
  diagram,
  \begin{displaymath}
    \xymatrix{
      \Hom_\ZZ(E,\QZ)\otimes_\B F \ar[r] \ar[d]_-{\cong} &
      \Hom_\ZZ(E,\QZ)\otimes_\B \B(B,-) \ar[d]^-{\cong} \\
      \Hom_\ZZ((\B,\Ab)(F,E),\QZ) \ar[r] &
      \Hom_\ZZ((\B,\Ab)(\B(B,-),E),\QZ).
    }
  \end{displaymath}
  The desired conclusion now follows from Remark \eqref{rmk:Baer}.
\end{proof}

\begin{prp}
  \label{prp:Bnoet-eqc}
  The category $\B$ is left noetherian if and only if it satisfies:
  \begin{rqm}
  \item $\B$ is left coherent, and
  \item Any functor $E$ is in $\Inj(\B,\Ab)$ if only if
    $\,\Hom_\ZZ(E,\QZ)$ is in $\Flat(\Bop,\Ab)$.
  \end{rqm}
\end{prp}

\begin{proof}
  ``If'': Let $\{E_\lambda\}$ be a family in $\Inj(\B,\Ab)$. Combining
  the isomorphism
  \begin{displaymath}
    \Hom_\ZZ(\textstyle{\coprod}E_\lambda,\QZ) \,\cong\,
    \textstyle{\prod}\Hom_\ZZ(E_\lambda,\QZ)
  \end{displaymath}
  with (1) and (2), it follows by Definition \eqref{dfn:cat_noeth}
  that $\B$ is left noetherian.

  ``Only if'': If $\B$ is left noetherian then (1) holds by Remark
  \eqref{rmk:coh-noet}, and the ``only if'' part of (2) follows from
  Proposition \eqref{prp:dual_to_injective}. Using that every additive
  subfunctor of $\B(B,-)$ is finitely presented, one proves the ``if''
  part of (2) by an argument similar to that found in the proof of
  Proposition \eqref{prp:dual_to_injective}.
\end{proof}

\providecommand{\bysame}{\leavevmode\hbox to3em{\hrulefill}\thinspace}
\providecommand{\MR}{\relax\ifhmode\unskip\space\fi MR }
\providecommand{\MRhref}[2]{%
  \href{http://www.ams.org/mathscinet-getitem?mr=#1}{#2}
}
\providecommand{\href}[2]{#2}


\begin{thebibliography}{10}

\bibitem{FWA-KRF-92}
F.~W. Anderson and K.~R. Fuller, \emph{Rings and categories of modules}, second
  ed., Grad. Texts in Math., vol.~13, Springer-Verlag, New York, 1992.

\bibitem{LAH00}
L.~Angeleri-H{\"u}gel, \emph{On some precovers and preenvelopes},
  Habilitationsschrift, M{\"u}nchen, 2000, ISBN 3-89791-137-X, available from
  \texttt{http://www.dicom.uninsubria.it/$\sim$langeleri/}.

\bibitem{LAH03}
\bysame, \emph{Covers and envelopes via endoproperties of modules}, Proc.
  London Math. Soc. (3) \textbf{86} (2003), no.~3, 649--665.

\bibitem{LAH-JS-JT-08}
L.~Angeleri-H{\"u}gel, J.~{\v{S}}aroch, and J.~Trlifaj, \emph{On the telescope
  conjecture for module categories}, J. Pure Appl. Algebra \textbf{212} (2008),
  no.~2, 297--310.

\bibitem{LAH-JT-04}
L.~Angeleri-H{\"u}gel and J.~Trlifaj, \emph{Direct limits of modules of finite
  projective dimension}, Rings, modules, algebras, and abelian groups, Lecture
  Notes in Pure and Appl. Math., vol. 236, Dekker, New York, 2004, pp.~27--44.

\bibitem{MA78}
M.~Auslander, \emph{Functors and morphisms determined by objects},
  Representation theory of algebras (Proc. Conf., Temple Univ., Philadelphia,
  Pa., 1976), Dekker, New York, 1978, pp.~1--244. Lecture Notes in Pure Appl.
  Math., Vol. 37.

\bibitem{MA-MB-69}
M.~Auslander and M.~Bridger, \emph{Stable module theory}, Mem. Amer. Math.
  Soc., no. 94, Amer. Math. Soc., Providence, R.I., 1969.

\bibitem{HB62}
H.~Bass, \emph{Injective dimension in {N}oetherian rings}, Trans. Amer. Math.
  Soc. \textbf{102} (1962), 18--29.

\bibitem{CE56}
H.~Cartan and S.~Eilenberg, \emph{Homological algebra}, Princeton Landmarks
  Math., Princeton Univ. Press, Princeton, NJ, 1999, With an appendix by David
  A. Buchsbaum, Reprint of the 1956 original.

\bibitem{SUC60}
S.~U. Chase, \emph{Direct products of modules}, Trans. Amer. Math. Soc.
  \textbf{97} (1960), 457--473.

\bibitem{LWC-AF-HH-06}
L.~W. Christensen, A.~Frankild, and H.~Holm, \emph{On {G}orenstein projective,
  injective and flat dimensions---a functorial description with applications},
  J. Algebra \textbf{302} (2006), no.~1, 231--279.

\bibitem{WCB94}
William Crawley-Boevey, \emph{Locally finitely presented additive categories},
  Comm. Algebra \textbf{22} (1994), no.~5, 1641--1674.

\bibitem{BE-AS-53}
B.~Eckmann and A.~Schopf, \emph{\"{U}ber injektive {M}oduln}, Arch. Math.
  \textbf{4} (1953), 75--78.

\bibitem{REB06}
R.~El~Bashir, \emph{Covers and directed colimits}, Algebr. Represent. Theory
  \textbf{9} (2006), no.~5, 423--430.

\bibitem{EEE87}
E.~E. Enochs, \emph{Minimal pure injective resolutions of flat modules}, J.
  Algebra \textbf{105} (1987), no.~2, 351--364.

\bibitem{EEE-OMGJ-95}
E.~E. Enochs and O.~M.~G. Jenda, \emph{Gorenstein injective and projective
  modules}, Math. Z. \textbf{220} (1995), no.~4, 611--633.

\bibitem{EEE-OMGJ-book}
\bysame, \emph{Relative homological algebra}, de Gruyter Exp. Math., vol.~30,
  Walter de Gruyter \& Co., Berlin, 2000.

\bibitem{EEE-OMGJ-JX-93}
E.~E. Enochs, O.~M.~G. Jenda, and J.~Xu, \emph{The existence of envelopes},
  Rend. Sem. Mat. Univ. Padova \textbf{90} (1993), 45--51.

\bibitem{EEE-OMGJ-96}
\bysame, \emph{Foxby duality and {G}orenstein injective and projective
  modules}, Trans. Amer. Math. Soc. \textbf{348} (1996), no.~8, 3223--3234.

\bibitem{EEE-SY-04}
E.~E. Enochs and S.~Yassemi, \emph{Foxby equivalence and cotorsion theories
  relative to semi-dualizing modules}, Math. Scand. \textbf{95} (2004), no.~1,
  33--43.

\bibitem{foxby:gmarm}
H.-B. Foxby, \emph{Gorenstein modules and related modules}, Math. Scand.
  \textbf{31} (1972), 267--284 (1973).

\bibitem{LF67}
L.~Fuchs, \emph{Algebraically compact modules over {N}oetherian rings}, Indian
  J. Math. \textbf{9} (1967), 357--374 (1968).

\bibitem{PG62}
P.~Gabriel, \emph{Des cat\'egories ab\'eliennes}, Bull. Soc. Math. France
  \textbf{90} (1962), 323--448.

\bibitem{GG04}
G.~Garkusha, \emph{Relative homological algebra for the proper class
  {$\omega\sb f$}}, Comm. Algebra \textbf{32} (2004), no.~10, 4043--4072.

\bibitem{golod:gdgpi}
E.~S. Golod, \emph{{$G$}-dimension and generalized perfect ideals}, Trudy Mat.
  Inst. Steklov. \textbf{165} (1984), 62--66, Algebraic geometry and its
  applications.

\bibitem{LG-CUJ-80}
L.~Gruson and C.~U. Jensen, \emph{Dimensions cohomologiques reli\'ees aux
  foncteurs {$\varprojlim\sp{(i)}$}}, Paul Dubreil and Marie-Paule Malliavin
  Algebra Seminar, 33rd Year (Paris, 1980), Lecture Notes in Math., vol. 867,
  Springer, Berlin, 1981, pp.~234--294.

\bibitem{HH04}
H.~Holm, \emph{Gorenstein homological dimensions}, J. Pure Appl. Algebra
  \textbf{189} (2004), no.~1-3, 167--193.

\bibitem{HH-PJ-06}
H.~Holm and P.~J{\o}rgensen, \emph{Semi-dualizing modules and related
  {G}orenstein homological dimensions}, J. Pure Appl. Algebra \textbf{205}
  (2006), no.~2, 423--445.

\bibitem{CUJ-HL-book}
C.~U. Jensen and H.~Lenzing, \emph{Model theoretic algebra with particular
  emphasis on fields, rings, modules}, Algebra Logic Appl., vol.~2, Gordon and
  Breach Science Publishers, New York, 1989.

\bibitem{RK67}
R.~Kie{\l}pi{\'n}ski, \emph{On {$\Gamma $}-pure injective modules}, Bull. Acad.
  Polon. Sci. S\'er. Sci. Math. Astronom. Phys. \textbf{15} (1967), 127--131.

\bibitem{HK01}
H.~Krause, \emph{The spectrum of a module category}, Mem. Amer. Math. Soc.
  \textbf{149} (2001), no.~707, x+125.

\bibitem{HK-MS-98}
H.~Krause and M.~Saor{\'{\i}}n, \emph{On minimal approximations of modules},
  Trends in the representation theory of finite-dimensional algebras (Seattle,
  WA, 1997), Contemp. Math., vol. 229, Amer. Math. Soc., Providence, RI, 1998,
  pp.~227--236.

\bibitem{HK-OS-03}
H.~Krause and {\O}.~Solberg, \emph{Applications of cotorsion pairs}, J. London
  Math. Soc. (2) \textbf{68} (2003), no.~3, 631--650.

\bibitem{DL64}
D.~Lazard, \emph{Sur les modules plats}, C. R. Acad. Sci. Paris \textbf{258}
  (1964), 6313--6316.

\bibitem{HL83}
H.~Lenzing, \emph{Homological transfer from finitely presented to infinite
  modules}, Abelian group theory (Honolulu, Hawaii, 1983), Lecture Notes in
  Math., vol. 1006, Springer, Berlin, 1983, pp.~734--761.

\bibitem{SML-book}
S.~Mac~Lane, \emph{Categories for the working mathematician}, second ed., Grad.
  Texts in Math., vol.~5, Springer-Verlag, New York, 1998.

\bibitem{JMM64}
J.-M. Maranda, \emph{Injective structures}, Trans. Amer. Math. Soc.
  \textbf{110} (1964), 98--135.

\bibitem{OR70}
U.~Oberst and H.~R{\"o}hrl, \emph{Flat and coherent functors}, J. Algebra
  \textbf{14} (1970), 91--105.

\bibitem{JR-MS-98}
J.~Rada and M.~Saorin, \emph{Rings characterized by (pre)envelopes and
  (pre)covers of their modules}, Comm. Algebra \textbf{26} (1998), no.~3,
  899--912.

\bibitem{BTS67}
B.~T. Stenstr{\"o}m, \emph{Pure submodules}, Ark. Mat. \textbf{7} (1967),
  159--171 (1967).

\bibitem{BS68}
\bysame, \emph{Purity in functor categories}, J. Algebra \textbf{8} (1968),
  352--361.

\bibitem{vasconcelos:dtmc}
W.~V. Vasconcelos, \emph{Divisor theory in module categories}, North-Holland
  Publishing Co., Amsterdam, 1974, North-Holland Mathematics Studies, no. 14,
  Notas de Matem\'atica No. 53. [Notes on Mathematics, No. 53].

\bibitem{CAW-book}
C.~A. Weibel, \emph{An introduction to homological algebra}, Cambridge Stud.
  Adv. Math., vol.~38, Cambridge Univ. Press, Cambridge, 1994.

\bibitem{JX96}
J.~Xu, \emph{Flat covers of modules}, Lecture Notes in Math., vol. 1634,
  Springer-Verlag, Berlin, 1996.

\end{thebibliography}

\end{document}